\documentclass{amsart}
\usepackage{amsmath,amsfonts,amssymb}
\usepackage{amsthm}
\usepackage{ifthen}
\usepackage{longtable}
\usepackage{array}
\usepackage{url}
\usepackage{color}
\usepackage{enumerate}

\newcommand{\Z}{\mathbb{Z}}
\newcommand{\Q}{\mathbb{Q}}
\newcommand{\R}{\mathbb{R}}
\newcommand{\C}{\mathbb{C}}
\newcommand{\N}{\mathbb{N}}
\newcommand{\FF}{\mathbb{F}}
\newcommand{\OO}{\mathcal{O}}

\newcommand{\NL}{\operatorname{N}_{L,k}}

\DeclareMathOperator{\Trace}{Tr}

\newtheorem{definition}{Definition}
\newtheorem*{theorem*}{Theorem}
\newtheorem{theorem}{Theorem}
\newtheorem{lemma}{Lemma}
\newtheorem{corollary}{Corollary}
\newtheorem{proposition}{Proposition}

\newtheorem{problem}{Problem}

\newcommand{\vc}{\mathbf{c}}
\newcommand{\vw}{\mathbf{w}}
\newcommand{\vu}{\mathbf{u}}
\newcommand{\vv}{\mathbf{v}}
\newcommand{\vxi}{\boldsymbol{\xi}}
\newcommand\NLk{\operatorname{N}_{L,k}}
\newcommand\eNLk{\tilde{\operatorname{N}}_{L,k}}
\newcommand\eNLz{\tilde{\operatorname{N}}_{L,2}}
\newcommand\card{\#}
\newcommand\where{\ :\ }
\newcommand\cF{\mathcal{F}}
\newcommand\cS{\mathcal{S}}
\newcommand\cU{\mathcal{U}}
\newcommand\cT{\mathcal{T}}
\newcommand\cE{\mathcal{E}}
\newcommand\vutil{\widetilde{\mathbf{u}}}
\newcommand\util{\widetilde{u}}
\newcommand\q{q}
\newcommand\qu{q}
\DeclareMathOperator{\WR}{R}
\DeclareMathOperator{\Spec}{Spec}
\newcommand{\A}{\mathbb{A}}
\newcommand{\cX}{\mathcal{X}}

\theoremstyle{remark}
\newtheorem{remark}{Remark}

\definecolor{dblackcolor}{rgb}{0.0,0.0,0.0}
\definecolor{dbluecolor}{rgb}{0.01,0.02,0.7}
\definecolor{dgreencolor}{rgb}{0.2,0.4,0.0}
\definecolor{dgraycolor}{rgb}{0.30,0.3,0.30}

\begin{document}
\title[Sums of $k$ units]{Rational integers as sums of units -- the quadratic case}

\author[C. Frei]{Christopher Frei}
\email{frei@math.tugraz.at}

\author[M. Widmer]{Martin Widmer}
\email{martin.widmer@tugraz.at}

\address{Institute of Analysis and Number Theory, Graz University of Technology,
  Steyrergasse 30/II, 8010 Graz, Austria}

\author[V. Ziegler]{Volker Ziegler}
\email{volker.ziegler@sbg.ac.at}
\address{University of Salzburg,
Hellbrunnerstrasse 34/I,
A-5020 Salzburg, Austria}

\begin{abstract}
How many natural numbers below $X$ can be written as a sum of $k$ units 
of the ring of integers of a given number field? We give the asymptotics
as $X$ gets large for quadratic number fields. This solves a problem of Jarden and Narkiewicz from 2007 for quadratic number fields.  
\end{abstract}

\maketitle

\section{Introduction}


Jarden and Narkiewicz proved that if $L$ is a number field then there is no natural number $k$ such that  every element of the ring of integers $\OO_L$ is a sum of at most $k$ units of $\OO_L$. 
More precisely they proved \cite[Corollary 6]{Jarden:2007} that the rational integers $n$ that are sums of at most $k$ units have density zero. 
Their proof is short and elegant, based on van der Warden's theorem and a classcial finiteness result concerning unit equations, but does not shed any light on the asymptotics of the counting function. 
They proposed the following problem \cite[Problem C]{Jarden:2007}.

\begin{problem}[Jarden and Narkiewicz, 2007]\label{prob:JN}
Let $L$ be a number field. Obtain an asymptotical formula for the number of positive rational integers $n\leq X$ which are sums of at most $k$ units of $\OO_L$. 
\end{problem}

So far this problem has not been addressed in the literature. In this article
we solve Problem \ref{prob:JN} for quadratic number fields.

For imaginary quadratic fields all units are roots of unity. Hence, no natural number $n>k$  is a sum
of at most $k$ units, whereas clearly all other $n$
are. 
So let us fix a \emph{real} quadratic number field $L=\Q(\sqrt{d})$ with $d\geq 2$ and
squarefree. For $\vw=(w_1,\ldots,w_r)\in L^r$, we write
\begin{equation}
  \label{eq:def_Sw}
  S_\vw:=w_1+\cdots+w_r.
\end{equation}
Throughout this paper, we let $X\geq 2$ and $k\in\N=\{1,2,3,\ldots\}$. 
We are interested in the set\footnote{We interpret $(\OO_L^\times)^0$ as containing only the empty tuple $\vu$, and $S_\vu=0$.}
$$\NL:=\{n\in\Z\where n=S_\vu \text{ for some }\vu\in(\OO_L^\times)^r\text{ with }0\leq r\leq k \}$$
and its counting function
\begin{equation}
  \NLk(X):=\card\{n\in\NL: |n|\leq X\}.
\end{equation}
Non-zero integers $n$ in $\NLk$ come in pairs $n,-n$. Hence,
$\NLk(X)-1$ is twice 
the number of positive rational integers $n\leq X$ which are sums of at most $k$ units of $\OO_L$.

Our main result is the following.

\begin{theorem}\label{thm:main}
  Let $\eta>1$ be the fundamental unit of the real quadratic field $L$, let
  $k\in\N$ and define $\rho:=\lfloor k/2\rfloor$. Then, for $X\geq 2$,
  \begin{equation*}
    \NLk(X)= c_k\left(\frac{2\log X}{\log\eta}\right)^\rho+O_{k,L}((\log X)^{\rho-1}),
  \end{equation*}
  where
  \begin{equation*}
    c_k=
    \begin{cases}
      1/\rho! &\text{ if $k$ is even,}\\
      3/\rho! &\text{ if $k$ is odd.}
    \end{cases}
  \end{equation*}
\end{theorem}
Although only very few rational integers $n$ are sums of at most $k$ units in a fixed real quadratc field $L$, 
every rational integer $n$ is the sum of two units in \emph{some} real quadratic field, e.g., for $|n|>2$ we can take the sum of conjugate units
$\frac{n\pm \sqrt{n^2-4}}{2}$.\\

Restricting to sums of \emph{exactly} $k$ units we define
$$\eNLk:=\{n\in\Z\where n=S_\vu \text{ for some }\vu\in(\OO_L^\times)^k\},$$
and its counting function
\begin{equation}
  \label{eq:def_NLk}
  \eNLk(X):=\card\{n\in\eNLk: |n|\leq X\}.
\end{equation}

That is, $\eNLk(X)$ is the number of integers $n$ with $|n|\leq X$ that can be written as the sum of exactly $k$ units. The following result is an immediate consequence of Theorem \ref{thm:main}.

\begin{corollary}\label{cor:exact_eqn}
  Let $\eta>1$ be the fundamental unit of the real quadratic field $L$, let
  $k\in\N$ and define $\rho:=\lfloor k/2\rfloor$. Then, for $X\geq 2$,
  \begin{equation*}
   \eNLk(X)= \tilde c_k\left(\frac{2\log X}{\log\eta}\right)^\rho+O_{k,L}((\log X)^{\rho-1}),
  \end{equation*}
  where
  \begin{equation*}
   \tilde c_k=
    \begin{cases}
      1/\rho! &\text{ if $k$ is even,}\\
      2/\rho! &\text{ if $k$ is odd.}
    \end{cases}
  \end{equation*}
\end{corollary}

Other aspects of the sets $\eNLk$ and $\NLk$, at least for $k=2$, have 
been studied before. 
Nagell asked for which number fields $L$ the number $1$ is contained in $\eNLz$. He called such number fields $L$ exceptional. Nagell's considerations culminated in \cite{Nagell:1969}, where he classified all exceptional number fields $L$ of unit rank $\leq 1$. More recently Freitas--Kraus--Siksek \cite{FKS:2021} have shown that, for any given prime $p \geq 5$, there are only finitely many cyclic degree $p$ fields $L$ that are exceptional. 

For cyclotomic fields $L = \Q(\zeta_p)$, Newman \cite[p.~89]{Newman:1974} observed that $1,2$ and $3$ are all contained in $\eNLz$  for all primes $p > 3$, and he posed the problem to explicitly determine $\operatorname{N}_{L,2}$  for cyclotomic number fields $L$.

Recently, Tinkova et.al. \cite{Tinkova:2025} considered the problem to completely determine the sets $\eNLz$ for cubic fields $L$. They resolved the problem for all cubic fields $L$ which are either cyclic or imaginary. 
Moreover, they showed that for number fields $L$ that do not contain a real
quadratic field the sets $\operatorname{N}_{L,2}$ are finite.

Moreover, quantities of similar spirit as $\operatorname{N}_{L,k}(X)$ were studied in \cite{MR2540792,MR3260868}.
\\

  There has been much activity recently and in the past regarding statistics
  for the number of fibres admitting rational or integral solutions in certain
  families of Diophantine equations. For an overview and references, we refer
  to the introduction of \cite{loughran2024leadingconstantrationalpoints}.

  Corollary \ref{cor:exact_eqn} can be interpreted in this vein, as counting
  asymptotically the number of fibres admitting integral points in a certain natural
  family of schemes parameterised by integers.

  More
  precisely, let $A:=\OO_L[t,x_0,\ldots,x_{k-1}]/(g)$, where
  \begin{equation*}
  g= x_0\cdots x_{k-1}(t-x_1-\cdots-x_{k-1})-1
  \end{equation*}
  and $\cX':=\Spec(A)$. The
  inclusion of $\OO_L[t]$ in $A$ induces a
  morphism $\cX' \to \A^1_{\OO_L}$ of schemes. The Weil
  restriction (see \cite{Ji_Li_McFaddin_Moore_Stevenson_2022})
  \begin{equation*}
    \cX := \WR_{\A^1_{\OO_L}/\A^1_{\Z}}(\cX')
  \end{equation*}
  comes with a morphism $\cX\to \A^1_\Z$. For every $n\in\Z$, we consider
  the pull-back $\cX_n:=\cX\times_{\A^1_\Z} \Spec(\Z)$ along the integral point
  $\Spec(\Z)\to\A^1_\Z$ induced by $\Z[t]\to \Z$, $t\to n$, which we call the
  \emph{fibre over $n$}. Using standard properties of the Weil restriction, we see that
  \begin{equation*}
    \cX_n \cong \WR_{\Spec(\OO_L)/\Spec(\Z)}(\cX'\times_{\A^1_{\OO_L}}\Spec(\OO_L)) =
    \WR_{\Spec(\OO_L)/\Spec(\Z)}(\Spec( A_n)),
  \end{equation*}
  where $A_n=A\otimes_{\OO_L[t]}\OO_L=\OO_L[x_0,\ldots,x_{k-1}]/(g_n)$ with
  \begin{equation*}
    g_n=x_0\cdots x_{k-1}(n-x_1-\cdots-x_{k-1})-1.
  \end{equation*}
  Hence, for every ring $B$, the set of $B$-points of $\cX_n$,
  \begin{equation*}
    \cX_n(B) = \Spec(A_n)(B\otimes_{\Z}\OO_L),
  \end{equation*}
  is in one-to-one correspondence with the set of solutions of the unit
  equation
  \begin{equation}\label{eq:B_points}
    u_1+\cdots+u_k=n
  \end{equation}
  with $u_i\in (B\otimes_\Z\OO_L)^\times$.
  Therefore, the function $\eNLk(X)$ studied in
  Corollary \ref{cor:exact_eqn} counts the set of integers $n\in\Z\cap [-X,X]$ for which
  the fibre $\cX_n$ of $\cX$ above $n$ has integral points,
  i.e. $\cX_n(\Z)\neq\emptyset$.  

  Regarding local solubility, when $k\geq 2$ it is clear that
  $\cX_n(\R)\neq \emptyset$, and straightforward to see that
  $\cX_n(\Z_p)\neq\emptyset$ whenever $p\neq 2$. Moreover,
  $\cX_n(\Z_2)\neq\emptyset$ if and only if $n\equiv k\bmod 2$ or $2$ is inert
  in $L$. Hence, in contrast to the global situation described in Corollary
  \ref{cor:exact_eqn}, a positive proportion of the fibres $\cX_n$ have points over
  $\R$ and all $\Z_p$. We give a detailed proof of the local solubility in Section \ref{sec:locsol}.\\

We end this introduction with a brief overview over the remaining sections. In Section \ref{sec:redtotracesums} we show that if $n$ is a sum of  at most $k$ units, then $n$ is a sum of traces 
of units and some summands from a finite set depending only on $k$ and $L$. 
Hence, we must count vectors $\vu$ with $\ell\leq\rho$
components of units whose trace sums have no vanishing subsums and are bounded in modulus by $X$. 
This is achieved in Section \ref{sec:tracesums}. Different unit vectors can lead to the same integer $n$ by permuting the coordinates of the vector, but also in more subtle ways. Counting these clashes is the purpose of Section \ref{sec:nonut}. In Section \ref{sec:proofthm}   
we are ready to prove Theorem \ref{thm:main}.
The final Section \ref{sec:locsol} is devoted to the 
proof of the claims about local solubility of (\ref{eq:B_points}). 

\section{Reduction to unit trace sums}\label{sec:redtotracesums}

For $\vu\in L^r$, we say that the sum $S_\vu=u_1+\cdots+u_r$ has \emph{no
  vanishing subsum}, if
\begin{equation*}
\sum_{i\in I}u_i\neq 0\quad \text{ for all }\quad \emptyset\neq I\subseteq\{1,\ldots,r\}, 
\end{equation*}
and \emph{no
  vanishing proper subsum}, if
\begin{equation*}
\sum_{i\in I}u_i\neq 0\quad \text{ for all }\quad \emptyset\neq I\subsetneq\{1,\ldots,r\}.
\end{equation*}

Moreover, we let $\vu'$ denote its conjugate. I.e., $\vu':=(u_1',\ldots,u_r')$,
where $u_i'=\sigma(u_i)$, with $\sigma:L\to L$ the non-trivial
$\Q$-automorphism.

At several places we will use the well-known fact that for any $T\in\N$ the unit equation
\begin{equation}\label{eq:unit_equation}
  v_1+\cdots+v_{T}=1,
\end{equation}
has only finitely many \emph{non-degenerate} solutions
$\vv\in(\OO_L^\times)^{T}$, i.e. solutions in which no subsum of the left-hand
side vanishes (e.g. \cite[Theorem 2A in Chapter V]{Schmidt:1991}). We denote the set
of these solutions, depending only on $L$ and $T$, by $\cS_{T}$. 

While for general number fields this is a deep fact (based on the subspace
theorem), proved by Evertse \cite{MR766298} as well as van der Poorten and Schlickewei
\cite{MR1119694}, we only need the case when $L$ is real quadratic, where it is a consequence of the simple Lemma \ref{lem:sum_lb}.

\begin{proposition}\label{prop:trace_sum_reduction}
  There is a chain of finite subsets $\cU_0\subseteq \cU_1\subseteq \cdots$ of
  $\OO_L^\times$, such that the following holds true.

  If $r\in\N$, $n\in\Z\smallsetminus\{0\}$ and $\vu\in(\OO_L^\times)^{r}$ such
  that $n=S_\vu$ with no vanishing subsum, then we also have
  $n=S_{(\vv,\vv',\vxi)}$, where
  \begin{equation*}
    \vv\in(\OO_L^\times)^{\ell}\ \text{ and }\ \vxi\in \cU_s^s,\quad\text{ with
    }\quad\ell,s\in \N_0 \ \text{ satisfying }\ 2\ell+s\leq r.
  \end{equation*}
 \end{proposition}

\begin{remark}\ 
  \begin{enumerate}
  \item We interpret $\vxi\in\cU_0^{0}$ as the empty tuple, so that
    $(\vv,\vv',\vxi)=(\vv,\vv')$.
  \item If $\vxi\in(\OO_L^\times)^1$ and $S_{(\vv,\vv',\vxi)}\in\Z$, then
    $\vxi\in\{\pm 1\}$.
  \end{enumerate}
\end{remark}

\begin{corollary}\label{cor:trace_sum_reduction}
  In the conclusion of Proposition \ref{prop:trace_sum_reduction}, we may
  additionally require that the sum $S_{(\vv,\vv',\vxi)}$ has no vanishing subsum.
\end{corollary}

\begin{proof}
  If $n=S_{(\vv,\vv',\vxi)}$ has a vanishing subsum, then the remaining
  summands form a tuple $\tilde\vu\in(\OO_L^\times)^{\tilde r}$ with $\tilde
  r\leq r-2$ and $n=S_{\tilde\vu}$. We apply Proposition
  \ref{prop:trace_sum_reduction} to $\tilde r$ and $\tilde\vu$ in place of $r$
  and $\vu$, which yields a
  representation $n=S_{(\vv_1,\vv_1',\vxi_1)}$. If the latter has no vanishing
  subsums, we are done. Otherwise, repeat the above, leading to an even shorter
  representation of the form $n=S_{(\vv_2,\vv_2',\vxi_2)}$. This process has to
  stop with a tuple $(\vv_s,\vv_s',\vxi_s)$ with at most $r$ coordinates, such
  that $n=S_{(\vv_s,\vv_s',\vxi_s)}$ has no vanishing subsums.
\end{proof}

\subsection{Proof of Proposition \ref{prop:trace_sum_reduction}}

\begin{lemma}\label{lem:trace_sum_reduction_base_case}
  Let $u_1,u_2\in \OO_L^\times$ with $u_1+u_2\in\Z$. Then one of the following
  three situations holds:
  \begin{enumerate}
  \item $u_2=u_1'$,
  \item $u_2=-u_1$,
  \item $\{u_1,u_2\}\in \left\{\left\{\frac{3\epsilon_1 +\sqrt{5}}{2},\frac{\epsilon_2-\sqrt{5}}{2}\right\},\left\{\frac{3\epsilon_1 -\sqrt{5}}{2},\frac{\epsilon_2+\sqrt{5}}{2}\right\}: \epsilon_1,\epsilon_2\in\{\pm 1\} \right\}$.
  \end{enumerate}
\end{lemma}
\begin{proof}
  As $u_1,u_2\in\OO_L$ with $u_1+u_2\in\Z$, we can write
  \begin{equation*}
    u_1=\frac{a_1+b\sqrt{d}}{2},\quad u_2=\frac{a_2-b\sqrt{d}}{2}
  \end{equation*}
  with $a_1,a_2,b\in\Z$. As $u_1,u_2$ are units, we have $u_1u_1'=\pm 1$ and
  $u_2u_2'=\pm 1$.

  If $u_1u_1'=u_2u_2'$, then $a_1^2-db^2=a_2^2-db^2$ and thus $a_2=\pm
  a_1$, yielding situation \emph{(1)} or \emph{(2)}.

  Now suppose that $u_1u_1'=-u_2u_2'$, and without loss of generality
  $u_1u_1'=1$. Then $a_1^2-db^2=4$ and $a_2^2-db^2=-4$, which implies that
  $a_1^2-a_2^2=8$, and thus $(a_1,a_2)=(\pm 3,\pm 1)$.

  This gives $9-db^2=4$, and thus $d=5,b=\pm 1$. Hence, we are in situation \emph{(3)}.
\end{proof}

We construct the sets $\cU_0\subseteq \cU_1\subseteq\cdots$ in Proposition
\ref{prop:trace_sum_reduction} as follows. Take $\cU_0=\cU_1:=\{\pm 1\}$ and
$\cU_2$ to consist of $\pm 1$ and
possibly the elements appearing in case \emph{(3)} of Lemma
\ref{lem:trace_sum_reduction_base_case}.

Now let $t\geq 3$ and assume that we have already constructed the sets
$\cU_0\subseteq\cdots\subseteq\cU_{t-1}$. For every non-degenerate solution
$(v_1,\ldots,v_{2t-1})\in \cS_{2t-1}$ of the unit equation
\eqref{eq:unit_equation} with $T=2t-1$, write
$\vv:=(v_1,\ldots,v_t)$. For each of the at most $\card\cS_{2t-1}$ choices of
$\vv$, there are at most two values of $u\in\OO_L^\times$ with
$u S_\vv\in \Z\smallsetminus\{0\}$, and we take $\cU_{t}$ to be the union of
$\cU_{t-1}$ with all coordinates $uv_i$ of all tuples $\vu:=u\vv$ as above.

Having described the sets $\cU_t$, we now prove Proposition
\ref{prop:trace_sum_reduction} by induction on $r$. For $r=1$, the conclusion
holds trivially. For $r=2$, it follows from Lemma \ref{lem:trace_sum_reduction_base_case}.

Hence, let $r\geq 3$ and assume that the proposition's conclusion holds for all
sums of less than $r$ terms.

From $n=S_\vu$, we see that also $n=S_{\vu'}$ and thus
\begin{equation*}
  0 = n-n = u_1+\cdots+u_r-u_1'-\cdots-u_r'.
\end{equation*}
Hence, there are subsets $I,J\subseteq \{1,\ldots,r\}$, such that
\begin{equation}
  \label{eq:minimal_vanishing_subsum}
  \sum_{i\in I}u_i - \sum_{j\in J}u_j'
\end{equation}
is a \emph{minimal} vanishing subsum, i.e. no proper subsum vanishes. As
$n=S_\vu$ has no vanishing subsums, we conclude that $I,J\neq
\emptyset$. Moreover, we may assume without loss of generality that $|I|\geq
|J|$, as conjugating and multiplying by $-1$ allow us to exchange the roles of
$I$ and $J$. We observe that then
\begin{equation}\label{eq:I_complement_rep}
  n=\sum_{i=1}^ru_i=\sum_{i=1}^ru_i-\left(\sum_{i\in I}u_i-\sum_{j\in
      J}u_j'\right) = \sum_{i\in I^c}u_i+\sum_{j\in J}u_j',
\end{equation}
where $I^c=\{1,\ldots,r\}\smallsetminus I$. We now distinguish between four
different cases.

\subsection*{Case 1: $|I|>|J|$}
As the sum on the right-hand side of \eqref{eq:I_complement_rep} has $|I^c|+|J| = r-|I|+|J|<r$ terms,
we find a minimal subsum with $\q<r$ terms which equals $n$. As the
subsum is minimal, it has no vanishing subsums. Hence, the induction hypothesis
yields a representation $n=S_{(\vv,\vv',\vxi)}$ with
$\vv\in(\OO_L^\times)^\ell$ and $\vxi\in \cU_{s}^s$, such that $2\ell+s\leq
\q< r$. This is enough for the proposition's conclusion to hold.
\subsection*{Case 2: $I^c\cap J\neq \emptyset$}

Let $j_0\in I^c\cap J$ and $m:=u_{j_0}+u_{j_0}'\in\Z$. Then by
\eqref{eq:I_complement_rep} we get the representation
\begin{equation*}
  n-m=\sum_{i\in (I\cup\{j_0\})^c}u_i + \sum_{J\smallsetminus\{j_0\}}u_j'.
\end{equation*}
There is a minimal subsum of the right-hand side that equals $n-m$, with  $\q\leq r-(|I|+1)+|J|-1\leq r-2$ terms. Again, the
induction hypothesis yields a representation $n-m=S_{(\vv_1,\vv_1',\vxi)}$
with $\vv_1\in(\OO_L^\times)^{\ell_1}$ and $\vxi_1\in \cU_s^{s}$, such that
$2\ell_1+s\leq \q\leq r-2$. 

Then we may take $\ell:=\ell_1+1$ and 
$\vv:=(\vv_1,u_{j_0})\in (\OO_L^\times)^\ell$,
giving $2\ell+s\leq r$ and $n=S_{(\vv,\vv',\vxi)}$ as desired.

\subsection*{Case 3: $I=J=\{1,\ldots,r\}$}

In this case, the sum $u_1+\cdots+u_r-u_1'-\cdots-u_r'=0$ has no vanishing proper subsums, and
hence so does the sum
\begin{equation}\label{eq:unit_sum_eq}
  v_1+\cdots+v_{2r-1}=1,
\end{equation}
where
\begin{equation*}
  v_i:=\frac{u_i}{u_r'}\ (1\leq i\leq r),\quad v_i:=\frac{-u_{i-r}'}{u_r'}\
  (r+1\leq i\leq 2r-1).
\end{equation*}
Hence, $(v_1,\ldots,v_{2r-1})\in \cS_{2r-1}$. Writing
$\vv:=(v_1,\ldots,v_r)$, then $\vu = u_r'\vv$, which implies that
$u_r'S_\vv=S_\vu=n\in\Z\smallsetminus\{0\}$. By construction of $\cU_r$, this
implies that all coordinates of $\vu=u_r'\vv$ are in $\cU_r$. Hence, the
proposition's conclusion is satisfied with $\ell=0$ and $\vxi=\vu\in\cU_r^r$.

\subsection*{Case 4: $I=J\subsetneq\{1,\ldots,r\}$}

In this case, we see from \eqref{eq:minimal_vanishing_subsum} that
\begin{equation*}
\sum_{i\in I}u_i =
\left(\sum_{u\in I}u_i\right)'
\end{equation*}
and thus $m:=\sum_{i\in I}u_i \in \Q\cap \OO_L = \Z$. As $S_\vu$ has no
vanishing subsums by hypothesis, we see that $m\notin\{0,n\}$, and also the above
representation of $m$ has no vanishing subsums. As $1\leq \q:=|I|<r$, the induction hypothesis
yields a representation $m=S_{(\vv_1,\vv_1',\vxi_1)}$ with
$\vv_1\in(\OO_L^\times)^{\ell_1}$ and $\vxi_1\in\cU_{s_1}^{s_1}$, such that
$2\ell_1+s_1\leq \q$. 

Moreover, we may write
\begin{equation*}
  n-m=\sum_{i=1}^ru_i-\sum_{i\in I}u_i = \sum_{i\in I^c}u_i,
\end{equation*}
again a representation without vanishing subsums, as $S_\vu$ has no vanishing
subsums. As $1\leq r-\q=|I^c|<r$, the induction hypothesis yields a
representation
$n-m=S_{(\vv_2,\vv_2',\vxi_2)}$ with $\vv_2\in(\OO_L^\times)^{\ell_2}$ and
$\vxi_2\in\cU_{s_2}^{s_2}$, such that $2\ell_2+s_2\leq r-\q$.

Then we may take $\ell=\ell_1+\ell_2$,
$s:=s_1+s_2$, $\vv=(\vv_1,\vv_2)\in(\OO_L^\times)^\ell$ and
$\vxi:=(\vxi_1,\vxi_2)\in \cU_{s}^s$ to obtain $2\ell+s\leq r$ and
$n=S_{(\vv,\vv',\vxi)}$.\qed

\section{Counting unit trace sums}\label{sec:tracesums}
Throughout this section we fix $\ell\in \N$. Let $\mathbf c=(c_1, \ldots, c_\ell) \in (L^\times)^\ell$. 
This section is devoted to study the following counting function 
\begin{equation*}
  T_{L,\ell}^{\mathbf{c}}(X):= \card\left\{(u_1, \ldots, u_\ell)\in\OO_L^\times \ :\
    \begin{aligned}
      &\left|\Trace_{L/\Q}(c_1u_1) + \cdots + \Trace_{L/\Q}(c_\ell u_\ell)\right|\leq X;\\
      &c_1u_1 + \cdots +c_\ell u_\ell+c'_1u'_1+\cdots+c'_\ell u'_\ell\\
      &\text{has no vanishing subsum};\\
      &|u_i|\geq 1 \text{ for }1\leq i \leq \ell.
    \end{aligned}
\right\}
\end{equation*}

If we drop the third condition $|u_i|\geq 1$ then we can replace  any of the coordinates $u_i$ by their conjugates $u'_i$ so that for each
$\vu$ counted in $T_{L,\ell}^{\mathbf{c}}(X)$ we have at most $2^\ell$ vectors. Hence, dropping the condition $|u_i|\geq 1$ gives a set of cardinality at most $2^\ell T_{L,\ell}^{\mathbf{c}}(X)$.

The main result of this section
provides an asymptotic formula for $T_{L,\ell}^{\mathbf{c}}(X)$ as $X$ gets large.

\begin{proposition}\label{prop:Counting-Tracesums}
For $X\geq 2$ we have 
$$T_{L,\ell}^{\mathbf{c}}(X)=\left(\frac{2\log X}{\log\eta}\right)^\ell+O_{L,\ell,\mathbf{c}}\left((\log X)^{\ell-1}\right). $$
\end{proposition}

\subsection{Proof of Proposition \ref{prop:Counting-Tracesums}}
We prove the upper and lower bound separately.
To prove the required upper bound we need
the following lemma. A different version was proved by the last author in \cite[Proposition 3.2]{Ziegler:2019}. 
\begin{lemma}\label{lem:sum_lb}
 Let $\qu\geq 1$ be an integer, $\alpha\in \C$ with $|\alpha|>1$, $\mathbf c=(c_1, \ldots, c_\qu) \in (\C^\times)^\qu$ and $n_1\geq \cdots \geq n_\qu$ integers.
 Then there exists $C=C(\alpha,\mathbf{c})>0$ depending only on $\alpha$ and $\mathbf c$, such that 
 \begin{alignat*}1
 \left|c_1\alpha^{n_1}+\cdots+c_\qu \alpha^{n_\qu}\right|>C\alpha^{n_1},
\end{alignat*}
 provided that $c_1\alpha^{n_1}+\cdots+c_\qu \alpha^{n_\qu}$ has no vanishing subsum.  
\end{lemma}
\begin{proof}
This is trivial for $\qu=1$.  
For $\qu\geq 2$ we need to show that 
\begin{alignat}1\label{ineq:fcam}
f_{\mathbf{c}, \alpha}(\mathbf{m}):=\left|c_1+c_2\alpha^{-m_2}+\cdots+c_k \alpha^{-m_\qu}\right|\geq C(\alpha,\mathbf{c})>0
\end{alignat}
for every integer vector $\mathbf{m}=(m_2,\ldots,m_\qu)$ with $0\leq m_2\leq \cdots \leq m_\qu$, provided
no subsum of $c_1+c_2\alpha^{-m_2}+\cdots+c_\qu \alpha^{-m_\qu}$ vanishes.
First suppose $\qu=2$. Set $M_0=(\log|2c_2/c_1|)/\log |\alpha|$, so that 
$$f_{\mathbf{c}, \alpha}(\mathbf{m})=|c_1+c_2\alpha^{-m_2}|\geq \left|\frac{c_1}{2}\right|$$
whenever $m_2\geq M_0$. On the other hand, by the non-vanishing subsum hypothesis, we have
$$\min_{0\leq m_2\leq M_0}f_{\mathbf{c}, \alpha}(\mathbf{m})=:C_0>0.$$
This proves the claim for $\qu=2$ with $C(\alpha,\mathbf{c})=\min\{|c_1/2|, C_0\}$.

Now let $\qu\geq 2$ be given and suppose (\ref{ineq:fcam}) holds, assuming the non-vanishing hypothesis. 
Let $c_{\qu+1}\in \C^\times$. Set 
$$M_1=\frac{\log\left|\frac{2c_{\qu+1}}{C(\alpha,\mathbf{c})}\right|}{\log |\alpha|},$$ 
and consider the integer $\qu$-tuple $(\mathbf{m},m_{\qu+1})$ with $0\leq m_2\leq \cdots \leq m_\qu\leq m_{\qu+1}$.
First suppose that $m_{\qu+1}\geq M_1$ then 
\begin{alignat*}1
f_{(\mathbf{c},c_{\qu+1}), \alpha}((\mathbf{m},m_{\qu+1}))&=|c_1+c_2\alpha^{-m_2}+\cdots +c_\qu\alpha^{-m_\qu}+c_{\qu+1}\alpha^{-m_{\qu+1}}|\\
&\geq f_{\mathbf{c}, \alpha}(\mathbf{m})-|c_{\qu+1}\alpha^{-M_1}|\\
&\geq \frac{C(\alpha,\mathbf{c})}{2}.
\end{alignat*}
Next suppose that $m_{\qu+1}\leq M_1$. Using the non-vanishing subsum hypothesis, we note that
\begin{alignat*}1
\min_{0\leq m_2\leq \cdots \leq m_{\qu+1}\leq M_1}|c_1+c_2\alpha^{-m_2}+\cdots +c_{\qu+1}\alpha^{-m_{\qu+1}}|=:C_1>0.
\end{alignat*}
Note that $C_1$ depends only on $\alpha, C(\alpha,\mathbf{c})$ and $c_{\qu+1}$.
Hence, we conclude 
\begin{alignat*}1
f_{(\mathbf{c},c_{\qu+1}), \alpha}((\mathbf{m},m_{\qu+1}))\geq \min\left\{\frac{C(\alpha,\mathbf{c})}{2},C_1\right\}
\end{alignat*}
for all integer $k$-tuple $(\mathbf{m},m_{\qu+1})$ with $0\leq m_2\leq \cdots \leq m_\qu\leq m_{\qu+1}$.
This completes the proof of the lemma. 
\end{proof}

Before we apply Lemma \ref{lem:sum_lb} to derive the required upper bound for 
$T_{L,\ell}^{\mathbf{c}}(X)$ let us point out that 
Lemma \ref{lem:sum_lb} also implies the finiteness
of the non-degenerate solutions $\vv\in(\OO_L^\times)^{T}$
to (\ref{eq:unit_equation}). For this it suffices to note that the conjugate $\vv'$ is also a solution of (\ref{eq:unit_equation}), so that we can assume
$n_1\geq |n_T|$.

\begin{lemma}\label{lem:ub-Tracesums}
 For $X\geq 2$ we have
 $$T_{L,\ell}^{\mathbf{c}}(X)\leq \left(\frac{2\log X}{\log\eta}\right)^\ell+O_{L,\ell,\mathbf{c}}\left((\log X)^{\ell-1}\right). $$
\end{lemma}
\begin{proof}
Let us consider an $\ell$-tuple $(u_1, \ldots, u_\ell)\in T_{L,\ell}^{\mathbf{c}}(X)$ so that
$$\left|\Trace_{L/\Q}(c_1u_1) + \cdots + \Trace_{L/\Q}(c_\ell u_\ell)\right|\leq X,$$
and no subsum of
$$c_1u_1 + \cdots +c_\ell u_\ell+c'_1u'_1+\cdots+c'_\ell u'_\ell$$
vanishes. Recall that $|u_i|\geq 1$, so that each $u_i$ has the form $u_i=\pm \eta^{n_i}$ with $n_i\in \N_0$. Applying Lemma \ref{lem:sum_lb} with $\alpha=\eta$ and $\qu=2\ell$ and taking logs gives
\begin{align*}
\max_{1\leq i\leq \ell}\{|n_i|\}\log\eta& \leq \log \left|\Trace_{L/\Q}(c_1u_1) + \cdots + \Trace_{L/\Q}(c_\ell u_\ell)\right|+O_{\mathbf{c}, L}(1)\\
& \leq \log X+O_{\mathbf{c}, L}(1).
\end{align*}
This immediately yields the upper bound
\begin{equation}\label{eq:up-bound-Tsum}
T_{L,\ell}^{\mathbf{c}}(X)\leq \left(\frac{2\log X}{\log\eta}\right)^\ell+O_{\mathbf{c}, L}\left((\log X)^{\ell-1}\right).
\end{equation}
\end{proof}

Next we prove the required lower bound for $T_{L,\ell}^{\mathbf{c}}(X)$.

\begin{lemma}\label{lem:lb-Tracesums}
 For $X\geq 2$ we have
 $$T_{L,\ell}^{\mathbf{c}}(X)\geq \left(\frac{2\log X}{\log\eta}\right)^\ell+O_{L,\ell,\mathbf{c}}\left((\log X)^{\ell-1}\right). $$
\end{lemma}
\begin{proof}
Set 
\begin{alignat*}1
C_1=C_1(\mathbf{c})=\frac{\max_i\{1,|c_i|,|c'_i|\}}{\min_j\{1,|c_j|,|c'_j|\}}, \text{ and }
C_2=\frac{\log(2\ell C_1)}{\log \eta}.
\end{alignat*}
It suffices to prove the bound for $X>2C_1\ell$.
Next let us count the $\mathbf{n}\in \N^\ell$ satisfying
\begin{itemize}
\item[1)] $\ell C_1 (\eta^{n_i}+1)< X$ $(1\leq i\leq \ell)$
\item[2)] $n_i\geq C_2$ $(1\leq i\leq \ell)$
\item[3)] $|n_i-n_j|\geq C_2$ $(1\leq i< j \leq \ell)$
\end{itemize}
The number of those $\mathbf{n}$ with 1) is 
\begin{alignat*}1
\left\lfloor \frac{\log(\frac{X}{C_1\ell}-1)}{\log \eta} \right \rfloor^\ell=\left(\frac{\log X}{\log\eta}\right)^\ell+O_{L,\ell,\mathbf{c}}\left((\log X)^{\ell-1}\right).
\end{alignat*}
And of those $\mathbf{n}$ only $O_{L,\ell,\mathbf{c}}\left((\log X)^{\ell-1}\right)$ fail 2) and only
$O_{L,\ell,\mathbf{c}}\left((\log X)^{\ell-1}\right)$ fail 3). Hence, we have 
\begin{alignat*}1
\left(\frac{\log X}{\log\eta}\right)^\ell+O_{L,\ell,\mathbf{c}}\left((\log X)^{\ell-1}\right)
\end{alignat*}
 $\mathbf{n}\in \N^\ell$ that satisfy 1), 2) and 3) simultaneously.
 Each of these $\mathbf{n}$ produces exactly $2^\ell$ unit vectors $\mathbf{u}\in {\OO_L^{\times}}^\ell$ with modulus of the coordinates $\geq 1$
 via $u_i=\pm \eta^{n_i}$ ($1\leq i\leq \ell$). Note that these 
 \begin{alignat*}1
\left(\frac{2\log X}{\log\eta}\right)^\ell+O_{L,\ell,\mathbf{c}}\left((\log X)^{\ell-1}\right)
\end{alignat*}
unit vectors $\mathbf{u}$ are pairwise distinct. We claim that all these 
unit vectors $\mathbf{u}$ are counted in $T_{L,\ell}^{\mathbf{c}}(X)$. First note that $|u'_i|=\eta^{-n_i}<1$,
and thus it follows from 1) that 
$$\left|\sum_{i=1}^\ell \Trace_{L/\Q}(c_i u_i)\right|<X.$$
Next suppose that 
$$\sum_{i=1}^\ell \Trace_{L/\Q}(c_i u_i)=c_1u_1+c'_1u'_1+\cdots +c_\ell u_\ell+c'_\ell u'_\ell$$
has a vanishing subsum, say
\begin{alignat}1\label{eq:zerosum}
v_1+\cdots +v_s=0
\end{alignat}
with $2\leq s\leq 2\ell$. After permuting the coordinates of $\mathbf{c}$ we can assume that $v_i= d_i \eta^{m_i}$
where $d_i\in \{\pm c_i,\pm c'_i\}$, and $m_1<m_2<\cdots <m_s$ are integers with $m_s-m_{s-1}\geq C_2$. The latter is a consequence of 2) and 3) for the positive integers $n_i$.
Dividing the zero-sum (\ref{eq:zerosum}) by $d_s$ yields
\begin{alignat*}1
\eta^{m_s}&=|(d_1/d_s)\eta^{m_1}+\cdots +(d_{s-1}/d_s)\eta^{m_{s-1}}|\\
&\leq (s-1)C_1\eta^{m_{s-1}}\\
&\leq (2\ell-1)C_1\eta^{m_{s}-C_2}\\
&= \frac{(2\ell-1)C_1}{2\ell C_1}\eta^{m_{s}}\\
&<\eta^{m_s}.
\end{alignat*}
This contradiction shows that no subsum vanishes, and therefore completes the proof of the lemma.
\end{proof}

Combining  Lemma \ref{lem:ub-Tracesums} and Lemma \ref{lem:lb-Tracesums} proves Proposition \ref{prop:Counting-Tracesums}

\section{Upper bounds for non-unique tuples}\label{sec:nonut}

\begin{definition}\label{def:equiv}
  We define an equivalence relation on $\Omega:=\bigcup_{m\in\N}L^m$ as
  follows: for $\vu\in L^m$ and $\vw\in L^n$, we have $\vu\sim\vw$ if and only
  if $n=m$ and $\vw$ arises from $\vu$ by a permutation of the coordinates.
\end{definition}

Let $\cU_t\subseteq\OO_L^\times$ be the finite subset from Proposition
\ref{prop:trace_sum_reduction} and define 
\begin{equation*}
  \cF_{t}:= \cU_0^0\cup\cU_1^1\cup\cdots\cup\cU_{t}^{t}.
\end{equation*}
Recall that we defined $\rho:=\lfloor k/2\rfloor$.
For $0\leq \ell\leq \rho$, we
consider the sets
\begin{equation}
  \cT_{k,\ell} := \left\{\vu=(\vv,\vv',\vxi)\where
  \begin{aligned}
    &\vv\in(\OO_L^\times)^\ell,\ \vxi\in\cF_{k-2\ell};\\
    &S_\vu\in\Z\text{ with no vanishing subsums;}\\
    & |v_i|\geq 1 \text{ for }1\leq i\leq \ell.
  \end{aligned}\right\}
\end{equation}
and $\cT_{k,\ell}(X):= \{\vu\in\cT_{k,\ell}\where |S_\vu|\leq X\}$. Next, we
define the subset of $\cT_{k,\rho}(X)$ of tuples $\vu$ that do
not represent $S_\vu$ essentially uniquely,
\begin{equation*}
  \cE_{k}(X):=\left\{\vu\in\cT_{k,\rho}(X)\where \exists \vutil\in
      \cT_{k,\rho}(X)\text{ such that }\vutil\not\sim\vu\text{ and }S_{\vutil}=S_\vu\right\}.
  \end{equation*}
  The main result of this section is an upper bound for the size of $\cE_k(X)$.

\begin{proposition}\label{prop:non_unique_bound}
  We have
  \begin{equation*}
    \card\cE_k(X)\ll_{k,L}(\log X)^{\rho -1}.
  \end{equation*}
\end{proposition}

\subsection{Proof of Proposition \ref{prop:non_unique_bound}}
Let $\vu\in(\OO_L^\times)^r$, $\vutil\in(\OO_L^\times)^s$ with $1\leq r,s\leq
k$, $\vu\not\sim\vutil$ and $S_\vu=S_{\vutil}=n\in\Z$, such that
both representations of $n$ have no vanishing subsums and $1\leq |n|\leq X$.

Then there are $I\subseteq\{1,\ldots,r\}$ and $J\subseteq\{1,\ldots,s\}$, such
that
\begin{equation}\label{eq:mvss_I_J}
  0=S_\vu-S_{\vutil}=\sum_{i\in I}u_i - \sum_{j\in J}{\util}_j
\end{equation}
is a \emph{minimal vanishing subsum}, i.e.  no subsum on the right-hand side
vanishes. As both $S_\vu$ and $S_{\vutil}$ have no vanishing subsums, it follows
that $I,J\neq\emptyset$.

If $|I|=|J|=1$, then we consider the complements
$I^c=\{1,\ldots,r\}\smallsetminus I$ and $J^c=\{1,\ldots,s\}\smallsetminus J$, 
and take a minimal vanishing subsum of
\begin{equation*}
  0=\sum_{i\in I^c}u_i-\sum_{j\in J^c}\util_j.
\end{equation*}
Continuing this way, we either find a minimal vanishing subsum of the form
\eqref{eq:mvss_I_J} with $|I|+|J|\geq 3$, or $\vu\sim{\vutil}$. As the latter was
excluded from the start, we may thus assume by symmetry that our minimal
vanishing subsum \eqref{eq:mvss_I_J} satisfies $|I|\geq 2$ and
$|J|\geq 1$.

For any $j_0\in J$, we thus have
\begin{equation*}
  \sum_{i\in I}\frac{u_i}{\util_{j_0}} - \sum_{j\in
    J\smallsetminus\{j_0\}}\frac{\util_j}{\util_{j_0}} = 1,
\end{equation*}
with no vanishing subsums. Write $u:=\util_{j_0}$ for simplicity, then
\begin{equation*}
  \vw := \frac{1}{u}((u_i)_{i\in I},(-\util_j)_{j\in J\smallsetminus\{j_0\}})\in
  (\OO_L^\times)^T,\quad T=|I|+|J|-1,
\end{equation*}
is a nondegenerate solution of the unit equation \eqref{eq:unit_equation}, whence
$\vw\in \cS_T$. Hence, for one of at most $\card\cS_T$ values of
$\vc=(c_i)_{i\in I}$, we have $u_i=c_iu$ for all $i\in I$.

We conclude that
\begin{equation*}
\vu \sim ((c_iu)_{i\in I},(u_i)_{i\in \{1,\ldots,r\}\smallsetminus I}).
\end{equation*}
As $|I|\geq 2$, we have decreased the number of free variables in
$\OO_L^\times$ by at least one, at the cost of introducing the coefficients
$\vc$.

For any $\vu\in\cE_k(X)$, we moreover know that $\vu\in\cT_{k,\rho}(X)$, and thus
\begin{equation*}
  ((c_iu)_{i\in I},(u_i)_{i\in \{1,\ldots,r\}\smallsetminus I}) \sim \vu =
  (\vv,\vv',\vxi)\text{ with }\vv\in(\OO_L^\times)^{\rho}\text{
    and }\vxi\in\cF_{k-\rho}.
\end{equation*}
Hence, if $k$ is even, then $r=k$ and $\vxi$ is the empty tuple. If $k$ is odd,
then either $r=k-1$ and $\vxi$ is the empty tuple, or $r=k$ and $\vxi\in\{\pm
1\}$.

As $|I|\geq 2$, there are $1\leq i_1<i_2\leq r$ with $\{i_1,i_2\}\subseteq I$,
and thus $u_{i_j}=c_{i_j}u$ for $j=1,2$.

Fixing $i_1,i_2$ and the $c_{i_j}$, we now distinguish a few different
cases, showing in each case that the number of tuples $\vu$ satisfying the
above conditions is $\ll_{k,L}(\log X)^{\rho-1}$.

\subsection*{Case 1: $i_2=2\rho +1$}
In this case, $k$ is odd and $i_2=r=k$. Therefore, $\vu=(\vv,\vv',\pm 1)$ and
one coordinate of $\vv$ (and $\vv'$) is also fixed, say with value $\pm a$,
where $a$ depends only on $c_{i_1}$ and $c_{i_2}$. Using Proposition
\ref{prop:Counting-Tracesums}, we get at most
\begin{equation*}
  \ll T_{L,\rho-1}^{(1,\ldots,1)}(X+1+|a+a'|)
  \ll_{k,L}(\log X)^{\rho-1}
\end{equation*}
possibilities for the value of $\vv$, and thus of $\vu$.

\subsection*{Case 2: $i_2=i_1+\rho\leq 2\rho $}
In this case, $c_{i_2}u = u_{i_2} = u_{i_1}'=c_{i_1}'u'$, and thus
\begin{equation*}
  u^2 = \pm\frac{u}{u'}=\pm\frac{c_{i_1}'}{c_{i_2}}.
\end{equation*}
This leaves only finitely many values of $u$, and thus of
$v_{i_1}=u_{i_1}=c_{i_1}u$. Fixing $v_{i_1}$ and using again Proposition
\ref{prop:Counting-Tracesums}, we get $\ll_{k,L}(\log X)^{\rho-1}$ choices for $\vv$, and thus for $\vu=(\vv,\vv',\vxi)$. 

\subsection*{Case 3: $i_2\leq  2\rho $ and $i_2\neq i_1+\rho$}
We may assume that $i_1< i_2\leq \rho$, possibly replacing $u_{i_j}$ by
$u_{i_j}'$. 
In any case, the coordinates $v_{i_1}$ and  $v_{i_2}$ of $\vv$ are both determined 
by $u$, and the sum of the traces of these coordinates, i.e. $(v_{i_1}+v'_{i_1})+(v_{i_2}+v'_{i_2})$, is the trace of $(c_{i_1}+c_{i_2})u$. Note that $c_{i_1}+c_{i_2}\neq 0$, as the $c_i$ are coordinates
of some $\vw\in\cS_T$.
Using Proposition \ref{prop:Counting-Tracesums}, we get that the number of $\vv$ (and thus also the number of $\vu$) is bounded by
\begin{equation*}
\ll  T_{L,\rho-1}^{(c_{i_1}+c_{i_2},1,\ldots,1)}(X+1)\ll_{k,L}(\log X)^{\rho-1}.
\end{equation*}
 \qed

\section{Proof of Theorem \ref{thm:main}}\label{sec:proofthm}
Recall Definition \ref{def:equiv} of the equivalence relation $\sim$ on
$\Omega=\bigcup_{m\in\N}L^m$. If $M\subseteq\Omega$, we write
\begin{equation*}
  M/\sim :=\{[m]\where m\in M\}
\end{equation*}
for the set of equivalence classes that have a representative in $M$. As the
set
$(A\smallsetminus B)/\sim$ clearly contains $(A/\sim)\smallsetminus (B/\sim)$, we see that
\begin{equation*}
  \NLk(X)\geq \card\left((\cT_{k,\rho}(X)\smallsetminus
    \cE_{k}(X))/\sim\right)\geq \card\left(\cT_{k,\rho}(X)/\sim\right)-\card\left(\cE_k(X)/\sim\right).
\end{equation*}
On the other hand, Proposition \ref{prop:trace_sum_reduction} and Corollary
\ref{cor:trace_sum_reduction} show that
\begin{equation*}
  \NLk(X)\leq \sum_{\ell=0}^{\rho}\card\left(\cT_{k,\ell}(X)/\sim\right).
\end{equation*}
Proposition \ref{prop:Counting-Tracesums} implies that
\begin{equation*}
  \cT_{k,\ell}(X)\ll_{L,k,\ell}(\log X)^\ell\quad\text{ for all }\quad 0\leq \ell\leq \rho,
\end{equation*}
by fixing $\vxi\in\cF_{k-2\ell}$ and counting all $\vv$ with
$|S_{(\vv,\vv')}|\leq X+|S_{\vxi} |\ll_{k,L}X$. Together with Proposition
\ref{prop:non_unique_bound}, this shows that
\begin{equation*}
  \NLk(X)=\card(\cT_{k,\rho}(X)/\sim)+O_{k,L}\left((\log X)^{\rho-1}\right).
\end{equation*}
Hence, it remains to evaluate $\card(\cT_{k,\rho}(X)/\sim)$
asymptotically. Elements $\vu\in\cT_{k,\rho}(X)$ have one of the following
shapes, all with $\vv\in(\OO_L^\times)^{\rho}$ and $|v_i|\geq 1$ for $1\leq i\leq \rho$:
\begin{enumerate}[($S_1$)]
\item $\vu=(\vv,\vv')$,
\item $\vu=(\vv,\vv',1)$,
\item $\vu=(\vv,\vv',-1)$.
\end{enumerate}
If $k$ is even, then only shape ($S_1$) is possible. If $k$ is odd, then all three
shapes can appear. Using that
\begin{equation*}
  |S_{(\vv,\vv',\vxi)}|-|S_{\vxi}| \leq |S_{(\vv,\vv')}|\leq  |S_{(\vv,\vv',\vxi)}|+|S_{\vxi}|, 
\end{equation*}
we see that in each of the three cases we have at least
$T_{L,\rho}^{(1,\ldots,1)}(X-1)$ and at most $T_{L,\rho}^{(1,\ldots,1)}(X+1)$
elements $\vu\in \cT_{k,\rho}(X)$. Hence, by Proposition
\ref{prop:Counting-Tracesums}, for $i\in\{1,2,3\}$ we have
\begin{equation*}
  \card\{\vu\in\cT_{k,\rho}(X)\where \vu \text{ of shape }(S_i)\} =
  \left(\frac{2\log X}{\log\eta}\right)^{\rho}+O_{k,L}((\log X)^{\rho-1}).
\end{equation*}
An easy application of Proposition \ref{prop:Counting-Tracesums} shows that the
contribution to the above count of those $\vu$ with two or more identical
coordinates is $\ll_{k,L}(\log X)^{\rho-1}$. 
Hence, we can assume the coordinates of $\vv$ are pairwise distinct and of modulus $>1$.
This means that for each $\vu$ in $\cT_{k,\rho}(X)$
there are exactly $\rho!$ many equivalent elements in
$\cT_{k,\rho}(X)$ (arising from permuting the first $\rho$
coordinates). We conclude that
\begin{equation*}
  \card\left(\{\vu\in\cT_{k,\rho}(X)\where \vu \text{ of shape }(S_i)\}/\sim\right) =
  \frac{1}{\rho!}\left(\frac{2\log X}{\log\eta}\right)^{\rho}+O_{k,L}((\log X)^{\rho-1}).
\end{equation*}
This proves Theorem \ref{thm:main} for even $k$, as then only shape
$(S_1)$ is possible.

If $k$ is odd, it only remains to note that all elements $\vu\in \cT_{k,\rho}$
that belong to the same equivalence class must share the same shape. Indeed,
the shape of $\vu$ is specified by the number of coordinates of $\vu$ and the parity of the number of coordinates of
$\vu$ equal to $1$, which are clearly constant in equivalence classes. Hence,
\begin{align*}
  \card\left(\cT_{k,\rho}(X)/\sim\right) &=
                                           \sum_{i=1}^3\card\left(\{\vu\in\cT_{k,\rho}(X)\where
                                           \vu \text{ of shape
                                           }(S_i)\}/\sim\right)\\ &=\frac{3}{\rho!}\left(\frac{2\log X}{\log\eta}\right)^{\rho}+O_{k,L}((\log X)^{\rho-1}).\qed
\end{align*}

\section{Local solubility}\label{sec:locsol}
    Let $p$ be a prime.
    We need to study the solubility of \eqref{eq:B_points} with $u_i\in (\Z_p\otimes_\Z\OO_L)^\times$. We
    start by investigating solutions with $u_i\in\Z_p^\times$.

    If either $k,n\in\Z_p^\times$ or $k,n\notin\Z_p^\times$, then at
    least one of $n/k$ and $(n-1)/(k-1)$ is in
    $\Z_p^\times$, and thus
    \begin{align*}
      \frac{n}{k}+\cdots+\frac{n}{k} = n \quad\text{ or }\quad 1+\frac{n-1}{k-1}+\cdots+\frac{n-1}{k-1}=n
    \end{align*}
    is a solution in units of $\Z_p$.

    If $p$ is odd and $p\mid k$, then $p\nmid n-e$ for some
    $e\in\{1,2\}$ and we get the solution
    \begin{equation*}
      e+\frac{n-e}{k-1}+\cdots+\frac{n-e}{k-1}=n,
    \end{equation*}
    in units of $\Z_p$.

    If $p$ is odd and $p\nmid k$, $p\mid n$, then also $p\nmid k-e$ for some
    $e\in\{1,2\}$ and we get the solution
    \begin{equation*}
      1+\frac{n-1}{k-1}+\cdots+\frac{n-1}{k-1} = n\quad\text{ or }\quad 1+1+ \frac{n-2}{k-2}+\cdots+\frac{n-2}{k-2}=n
    \end{equation*}
    in units of $\Z_p$.

    In conclusion, there are solutions in units of $\Z_p$ whenever $p$ is odd
    or $p=2$ and $n\equiv k\bmod 2$. Via $u\mapsto u\otimes 1$, these also give
    solutions in units of $\Z_p\otimes_\Z\OO_L$.

    If $p=2$ is inert in $\OO_L$, then $\Z_2\otimes_\Z\OO_L=\OO_{\mathfrak{P}}$, the localisation of $\OO_L$
    at the unique prime ideal $\mathfrak{P}$ over $2$. One easily sees that
    every element of $\FF_4=\OO_L/\mathfrak{P}$ can be written as a sum of two units, and hence
    \eqref{eq:B_points} has solutions over $\FF_4$, in units
    $u_i\in\FF_4^\times$. By Hensel's lemma, these solutions lift to solutions
    over $\OO_{\mathfrak{P}}$, still in units
    $u_i\in\OO_{\mathfrak{P}}^\times$. Hence, we have shown that
    $\cX_n(\Z_p)\neq\emptyset$ whenever $p\neq 2$, when $p=2$ is inert in $L$,
    or when $p=2$ and $k\equiv n\bmod 2$.

    When $k\not\equiv n\bmod 2$ and $2$ is split or ramified in $L$, let
    $\mathfrak{P}$ be a prime ideal of $\OO_L$ lying above $2$. From the
    reductions $\Z_2\to\FF_2$ and $\OO_L\to \OO_L/\mathfrak{P} =\FF_2$ we get a
    ring homomorphism
    \begin{equation*}
      \Z_2\otimes_\Z\OO_L\to \FF_2.
    \end{equation*}
    As under our hypothesis on $k$ and $n$ there are clearly no solutions of
    \eqref{eq:B_points} in units of $\FF_2$, this shows that there can be no such solutions in
    units of $\Z_2\otimes_\Z\OO_L$, and hence $\cX_n(\Z_2)=\emptyset$.

\section*{Acknowledgement}

This work was initiated at the special semester ``Heights in Diophantine Geometry, Group Theory and Additive Combinatorics'' held in at the ESI 
(Erwin Schrödinger International Institute for Mathematical Physics), where all three authors used the friendly and inspiring atmosphere to discuss 
the problem of Jarden and Narkiewicz \cite[Problem C]{Jarden:2007}. 
V.Z. was supported by the Austrian Science Fund (FWF) under the project P~24801-N26.

\bibliographystyle{abbrv}
\bibliography{Unitsumbib}
\end{document}